\documentclass[11pt]{amsart}
\usepackage[centering,margin=2.5cm,a4paper]{geometry}
\usepackage{amssymb}
\usepackage{tikz}
\usepackage{enumerate}
\usepackage{color}
\usepackage{soul}
\theoremstyle{plain}
\newtheorem{theorem}{Theorem}[section]
\newtheorem{lemma}[theorem]{Lemma}
\newtheorem{corollary}[theorem]{Corollary}
\newtheorem{proposition}[theorem]{Proposition}

\theoremstyle{definition}

\theoremstyle{remark}

\renewcommand{\le}{\leqslant}
\renewcommand{\leq}{\leqslant}
\renewcommand{\ge}{\geqslant}
\renewcommand{\geq}{\geqslant}

\newcommand{\R}{\mathbb{R}}

\title[Codimension two and three Kneser Transversals]{Codimension two and three Kneser Transversals}
\thanks{Supported by the ECOS-Nord project M13M01, by CONACyT project 166306, by CONACyT grant 277462 and by PAPIIT-UNAM project IN112614. The project leading to this application has received funding from European Research Council (ERC) under the European Union’s Horizon 2020 research and innovation programme under grant agreement No. 678765}
\thanks{$^{\star}$ Corresponding Author: jonathan.chappelon@umontpellier.fr}
\author[\tiny J.~Chappelon]{J.~Chappelon$^{\, \star}$}
\address{Institut Montpelli\'{e}rain Alexander Grothendieck, CNRS, Univ. Montpellier, France}
\email[Corresponding author]{jonathan.chappelon@umontpellier.fr}
\author[L.~Mart\'{i}nez-Sandoval]{L.~Mart\'{i}nez-Sandoval}
\address{Instituto de Matem\'{a}ticas, Universidad Nacional Aut\'{o}noma de M\'{e}xico, Ciudad Universitaria, M\'{e}xico D.F., 04510, Mexico and Dept. of Computer Science, Faculty of Natural Sciences, Ben-Gurion University of the Negev, Beer Sheva, 84105, Israel}
\email{leomtz@im.unam.mx}
\author[L.~Montejano]{L.~Montejano}
\address{Instituto de Matem\'{a}ticas, Universidad Nacional Aut\'{o}noma de M\'{e}xico, Ciudad Universitaria, M\'{e}xico D.F., 04510, Mexico}
\email{luis@math.unam.mx}
\author[L.P.~Montejano]{L.P.~Montejano}
\address{Institut Montpelli\'{e}rain Alexander Grothendieck, CNRS, Univ. Montpellier, France}
\email{lpmontejano@gmail.com}
\author[J.L.~Ram\'{i}rez Alfons\'{i}n]{J.L.~Ram\'{i}rez Alfons\'{i}n}
\address{Institut Montpelli\'{e}rain Alexander Grothendieck, CNRS, Univ. Montpellier, France}
\email{jorge.ramirez-alfonsin@umontpellier.fr}
\date{November 14, 2017}
\begin{document}
\begin{abstract}
Let $k,d,\lambda \geqslant 1$ be integers with $d\geqslant \lambda $ and let $X$ be a finite set of points in $\mathbb{R}^{d}$. 
A $(d-\lambda)$-plane $L$ transversal to the convex hulls of all $k$-sets of $X$ is called \emph{Kneser transversal}. If in addition $L$ contains $(d-\lambda)+1$ points of $X$, then $L$ is  called \emph{complete Kneser} transversal.
\smallskip

In this paper, we present various results on the existence of (complete) Kneser transversals for $\lambda =2,3$.  In order to do this, we introduce the notions of {\em stability} and {\em instability} for (complete) Kneser transversals.
We first give a stability result for collections of $d+2(k-\lambda)$ points in $\mathbb{R}^d$ with $k-\lambda\geqslant 2$ and $\lambda =2,3$. We then present a description of Kneser transversals $L$ of collections of $d+2(k-\lambda)$ points in $\mathbb{R}^d$ with $k-\lambda\geqslant 2$ for $\lambda =2,3$. We show that either $L$ is a complete Kneser transversal or it contains $d-2(\lambda-1)$ points and the remaining $2(k-1)$ points of $X$ are matched in $k-1$ pairs in such a way that $L$ intersects the corresponding closed segments determined by them. The latter leads to new upper and lower bounds (in the case when $\lambda =2$ and $3$)  for
$m(k,d,\lambda)$ defined as the maximum positive integer $n$ such that every set of $n$ points (not necessarily in general position) in $\mathbb{R}^{d}$ admit a Kneser transversal.
\smallskip

Finally, by using oriented matroid machinery, we present some computational results (closely related to the stability and unstability notions). We determine the existence of (complete) Kneser transversals for each of the $246$ different order types of configurations of $7$ points in $\mathbb{R}^3$.
\end{abstract}
\maketitle
\section{Introduction}
Let $k,d,\lambda \geq 1$ be integers with $d\geq \lambda $ and let $X\subset\R^{d}$ be a finite set. A $(d-\lambda)$-plane $L$ transversal to the convex hulls of all $k$-sets of $X$ is called \emph{Kneser transversal}. If in addition $L$ contains $(d-\lambda)+1$ points of $X$, then $L$ is  called \emph{complete Kneser} transversal.
\smallskip

Let $m(k,d,\lambda)$ be the maximum positive integer $n$ such that every set of $n$ points (not necessarily in general position) in $\mathbb{R}^{d}$ admits a Kneser Transversal.
\smallskip

This function was introduced in \cite{ABMR} where it was proved that

\begin{equation}\label{ineq1}
 d-\lambda +k+\left\lceil \frac{k}{\lambda }\right\rceil -1 \le m(k,d,\lambda )< d+2(k-\lambda)+1.
\end{equation}

The proof of the lower bound follows the same spirit of Dol'nikov \cite{D} and uses Schubert calculus in the cohomology ring of Grassmannian manifolds. The value of $m(k,d,\lambda )$ is strongly connected with  Rado's centerpoint theorem \cite{Rado}.  Indeed, let $n,d,\lambda \geq 1$ be integers with $d  \geq \lambda $ and let

\smallskip
\noindent {\bf $\tau(n,d,\lambda )$}$\overset{\mathrm{def}}{=}$ the maximum positive integer $\tau$ such that for any collection $X$ of $n$ points in $\mathbb{R}^{d}$, there is a  $(d-\lambda )$-plane $L_X$ such that any closed half-space $H$ through $L_X$ contains at least $\tau$ points.
\smallskip

We thus have that $n-\tau(n,d,\lambda)+1$ is equal to the minimum positive integer $k$ such that for any collection $X$ of $n$ points in $\mathbb{R}^{d}$  there is a common transversal $(d-\lambda )$-plane to the convex hulls of all $k$-sets { which turns out to be smaller than or equal to $\lfloor\frac{\lambda(n-d+\lambda)}{\lambda+1}\rfloor+1$. The latter was proved in \cite[Theorem 2]{ABMR} by using the lower bound of $m(k,d,\lambda)$ given in \eqref{ineq1} and therefore any improvement to the bounds for $m(k,d,\lambda)$ might shed light on the behaviour of $\tau(n,d,\lambda )$.}
\medskip

The case when $\lambda=1$ is of particular interest. In \cite{ABMR} it was proved that  $m(k,d,1)=d+2k-2$ and it was showed that this equality is equivalent to the fact that the chromatic number of the {\em Kneser graph} $KG(n,k)$ is $n-2k+2$, the well-known Kneser's conjecture originally proved by Lov\'asz \cite{L}.
\medskip

One of the purposes of this paper is to improve upper and lower bounds for $m(k,d,2)$ and $m(k,d,3)$, when $d$ is not too large. 
\smallskip

From the inequalities in \eqref{ineq1} it can be deduced that

\begin{equation}\label{eq2}
m(k,d,\lambda ) = d-\lambda +k+\left\lceil \frac{k}{\lambda }\right\rceil -1 \text{ for } \lambda=1, k-\lambda\leq 1 \text{ and } k\leq 3.
\end{equation}

Furthermore, the equality also holds for $d=\lambda$ \cite[Theorem 6]{ABMR}. In this paper, we shall focus our attention on the case when $d >\lambda \geq 2$  and  $k-\lambda\geq 2$.
\medskip

In \cite{CMMMR},  we introduced and studied a natural {\em discrete} version of the function $m(k,d,\lambda)$ in which it is asked the existence of complete Kneser transversals (the corresponding function is denoted by $m^*(k,d,\lambda)$). It turns out that the existence of a Kneser transversal is not necessarily an invariant of the {\em order type}. For example, for $d=2$, let $X$ be the vertex set of a regular hexagon. Then, the center of such hexagon is a $0$-plane transversal to the convex hull of the $4$-sets. But, by a suitable 'slight' perturbation of the $6$ vertices of the hexagon we lose this property, see Figure \ref{figura1}.
 \begin{figure}[htb]
 \includegraphics[width=.5\textwidth]{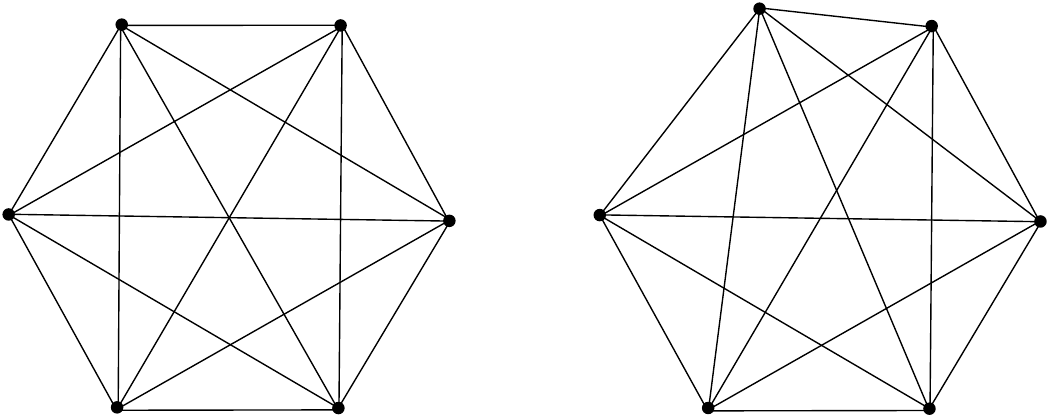}
 \caption{} \label{figura1}
\end{figure}
The situation for complete Kneser transversals is different. Indeed, in \cite[Section 2]{CMMMR} we showed how to detect complete Kneser transversal by using only Radon partitions implying that the existence of such transversals
 is an invariant of the order type. This naturally lead us to consider the notions of {\em stability} and {\em instability}. A Kneser transversal is said to be {\em stable} (resp. {\em unstable}) if the given set of points can be {\em slightly perturbed} (move each point to, not more than $\epsilon>0$ distance of their original position) such that the new configuration of points admits (if there is any) only complete Kneser transversals (resp. the new configuration of points does not admit a Kneser transversal).
\medskip

In the next section, we give a stability result for collections of $d+2(k-\lambda)$ points in $\R^d$ with $k-\lambda\ge 2$ for $\lambda =2,3$ (Theorem~\ref{stability}). In order to do this, we present a description of Kneser transversals $L$ for collections of $d+2(k-\lambda)$ points in $\R^d$ with $k-\lambda\ge 2$ and $\lambda =2,3$. We show that either $L$ is a complete Kneser transversal or it contains $d-2(\lambda-1)$ points and the remaining $2(k-1)$ points have a nice geometric property (Theorem~\ref{th:kt}).
\medskip

In Section~\ref{sec:bound}, we give an upper bound (Theorem~\ref{th:bound}) when $\lambda=2,3$, $(k-\lambda)\geq 2$ and $2(\lambda-1)\leq d\leq 2(k-2)$. Also, by using results due to Bukh, Matou\v sek and Nivasch \cite{BMN}, we obtain a lower bound for $m(k,d,2)$ (see Equation~\eqref{eq3}), which for $d\geq3, \lambda=2$ and $k$ large enough is better than the lower bound given in (1). 

\medskip

Finally, in Section~\ref{computer}, we present some computational results concerning  the existence of (complete) Kneser transversal lines to tetrahedra in configurations of $7$ points in $\mathbb{R}^{3}$. This provides a further insight on the relation between Kneser transversals and complete Kneser transversals. This is done by using oriented matroid machinery. We show that complete Kneser transversal lines to the tetrahedra can be completely determined by the oriented matroid associated to the corresponding configuration of 7 points. We then consider oriented matroids associated to some {\em extended} configurations of 8 points to detect the existence of non-complete Kneser transversals lines to tetrahedra (if any). This provides a method to determine the existence of Kneser transversals of configurations of 7 points in $\R^3$. We use these results to determine the existence of (complete) Kneser transversals for each of the 246 different order types of configurations of 7 points in $\R^3$ (Theorems \ref{th:comp1} and \ref{th:comp2}).

\section{Stability}

Our main result in this section is the following theorem.

\begin{theorem}\label{stability}
Let  $X=\{x_1,\dots ,x_n\}$ be a finite collection of points in $\R^{d}$. Suppose that $n=d + 2(k-\lambda)$, $k-\lambda \geq 2$, $\lambda=2,3$ and $d\geq \lambda$. For every $\epsilon >0$  there exists $X^\prime=\{x^\prime_1,\dots ,x^\prime_n\}$, a collection of points in $\R^{d}$  in general position such that $| x_i - x^\prime _i| < \epsilon$, for every $i=1,\dots ,n$, and with the property that every transversal $(d-\lambda)$-plane to the convex hull of the $k$-sets of $X^\prime$ is complete (i.e.,  it contains  $d-\lambda +1$ points of $X^\prime)$.
\end{theorem}

In order to prove Theorem \ref{stability}, we first need some results.

\begin{lemma}\label{lem1}
Suppose $Z=\{z_0,z_1,\dots,z_m\}\subset  \mathbb{R}^{m}$ is a collection of $m+1$ points in general position and $\ell$ is a codimension two transversal plane to the convex hull of all the $3$-sets of $Z$,  $m\geq3$.  Then  either $|\ell\cap Z|\geq m-2$ or  $|\ell\cap Z|= m-3$ and $\ell$ is transversal to two opposite edges of the tetrahedron with vertices in $Z-\ell$.  
\end{lemma}

\begin{proof}
If $m=3$ and $\ell$ is a transversal line  to the convex hull of all the $3$-sets of $Z=\{z_0,z_1,z_2,z_3\}$, then either $|\ell\cap Z|\geq 1$ or if $\ell\cap Z = \emptyset$, then  $\ell$ is transversal to two opposite edges of the tetrahedron with vertices $\{z_0,z_1,z_2,z_3\}$.

Let us prove now that if  $m=4$ and $\ell$ is a transversal $2$-plane to the convex hull of all the $3$-sets of $Z$ then $\ell\cap Z \not= \emptyset$. Suppose that $\ell\cap Z = \emptyset$ and without loss of generality assume that the intersection of $\ell$ with the $3$-plane $H$ generated by $Z-\{z_0\}$ is a line. Then, by the above $\ell\cap H$ is a line transversal to two opposite edges of the tetrahedron with vertices in $\{z_1,z_2,z_3, z_4\}$, let say $\{z_1,z_2\}$ and $\{z_3, z_4\}$. Note than $\ell$ does not intersect the four edges $\{z_1,z_3\}$, $\{z_3,z_2\}$, $\{z_2,z_4\}$, $\{z_4,z_1\}$, but $\ell$ must intersect the four triangles $\{z_0,z_1,z_3\}$, $\{z_0,z_3,z_2\}$, $\{z_0, z_2,z_4\}$, $\{z_0,z_4,z_1\}$, which is impossible unless $z_0\in\ell$.

Now it is easy to prove by induction that if $m\geq4$ and $\ell$ is a codimension two transversal plane to the convex hull of all the $3$-sets of $Z$ then $\ell\cap Z \not= \emptyset$ and using this to prove again by induction our lemma.
\end{proof}

\begin{lemma}\label{lem3}
Let $X\subset \R^{d}$ be a collection of $d+2$ points in general position, $d\geq2(\lambda-1)$ and $\lambda\geq 1$. Suppose $L_1$ and $L_2$ are $(d-\lambda$)-planes with the following properties:
\begin{enumerate}
\item
$L_1\cap X=L_2\cap X$ consist of $d-2(\lambda-1)$ points,
\item
The points of  $L_i - X$ can be matched in $\lambda$ pairs $\{z_1, z_2\},\dots \{z_{2\lambda-1}, z_{2\lambda}\}$ in such a way that both $L_1$ and $L_2$ intersects the corresponding lines determined by them.
\end{enumerate}
Then $L_1=L_2$.
\end{lemma}

\begin{proof} The proof is by induction on $\lambda$ projecting orthogonally  in the direction of the line through $\{z_1, z_2\}$. Note that the fact that $L_i$ intersects the line generated by $\{z_{2j-1},z_{2j}\}$, $j=2,\dots,\lambda$ implies that the orthogonal projection of $X$ in the direction of $z_2-z_1$ is a collection of  $d+1$ points in general position and the orthogonal projection of $L_i$ is a $(d-\lambda)$-plane of codimension $\lambda-1$, $i=1,2$.
\end{proof}

The following technical lemma will be crucial for the codimension three case in the main stability result.

\begin{lemma}\label{lem2}
Let $X \subset \mathbb{R}^{3}-\{0\}$  be a finite set of points and let $\Theta$ be a collection of triples of $X$ satisfying:
\begin{enumerate}[(i)]
\item
for every triple $\{x,y,z \} \in \Theta$ the triangle with vertices $\{x,y,z \}$ contains the origin in its interior,
\item
the intersection of any two triples of $\Theta$ contains at most one point.
\item
for every $x \not= y \in X$ there is $T \in \Theta$ such that  $\{x,y \} \subset T$.
\end{enumerate}
Then, $X$ is contained in a 2-plane.
\end{lemma}

\begin{proof}
Let $\{a,b,c\} \in \Theta$ and let $H$ be the 2-plane through the origin containing $\{a,b,c\}$. Let us suppose that $X$ is not contained in $H$. Let $G\subset X$ be the points of $X$ lying on one side of $H$ and $G^\prime$ the points of $X$ on the other side of $H$. By $(ii)$, there is a bijection $f:G \to G^\prime $ such that for every point of $x\in G$ there is a point $f(x) \in G^\prime$ with $\{a,x,f(x)\}\in \Theta$. 

Let $n=|G|=|G^\prime|$. If $n=1$ let $G=\{x\}$ and $G^\prime=\{x^\prime\}$. By iii), $\{x,a,x^\prime\}\in\Theta$, but also $\{x,b,x^\prime\}\in\Theta$, contradicting ii).  If $n=2$,  let $G=\{x,y\}$ and $G^\prime=\{x^\prime,y^\prime\}$. By iii) we may assume without loss of generality that $\{x,y,x^\prime\}\in\Theta$, 
 {but we also also have $\{x^\prime, y^\prime,w\}\in\Theta$, where $w\in \{x,y\}$, contradicting  ii).}
  
  We thus suppose that $n>2$. By $(i)$ and $(ii)$, for every $\{x,y \}\subset G$ there is $\psi(x,y)\in G^\prime$ such that $\{x,y,\psi(x,y) \}\in \Theta$. For every pair  $\{x,y \}$ of $G$ consider the two edges $(x,\psi(x,y) )$ and $(y,\psi(x,y))$. Since the intersection of any two triples of $\Theta$ contains at most one point we have that if  $\{u,v \}\not= \{x,y \}, u,v\in G$, then the four edges  $$(x,\psi(x,y) ), (y,\psi(x,y)), (u,\psi(u,v) ), (v,\psi(u,v))$$ are different.  Similarly, for every $\{x',y' \}\subset G^\prime$ there is $\psi(x',y')\in G$ such that $\{x',y',\psi(x',y') \}\in \Theta$. If we consider two different pairs of $G^\prime$,   $\{u',v' \}\not= \{x',y' \}$, then the four edges  $$(x',\psi(x',y') ), (y',\psi(x',y')), (u',\psi(u',v') ), (v',\psi(u',v'))$$ are different. Furthermore, if $\{x^\prime,y^\prime \}$ is a pair of $G^\prime$ and $\{x,y \}$ is a pair of $G$, then the four edges  $$(x,\psi(x,y) ), (y,\psi(x,y)), (x^\prime,\psi(x^\prime,y^\prime) ), (y^\prime,\psi(x^\prime,y^\prime))$$ are different. This implies that there are at least  $4{ n \choose 2 }$ edges between $G$ and $G^\prime$  {since we might have} two edges for every pair in $G$ and two edges for every pair in $G^\prime$, which is impossible since $4{ n \choose 2 }>n^2$ ($=$ the number of edges between $G$ and $G'$) for $n >2$.  Hence $X\subset H$.
\end{proof}

We may now determine (complete) Kneser transversals for codimensions 2 and 3.

\begin{theorem}\label{th:kt}
Let $X=\{x_1, x_2, \dots x_n\}$  be a collection of $n=d+2(k-\lambda)$ points in general position in $\mathbb{R}^{d}$. Suppose that $L$ is a $(d-\lambda)$-plane transversal to the convex hulls of all $k$-sets of $X$ with $\lambda=2, 3$ and $k\geq \lambda +2$ and $d\geq \lambda$.  Then, either
\begin{enumerate}[(1)]
\item
$L$ is a complete Kneser transversal (i.e., it contains $d-\lambda+1$ points of $X$) or
\item
$|X\cap L|= d-2(\lambda-1)$ and the other $2(k-1)$ points of $X$ are matched in $k-1$ pairs in such a way that $L$ intersects the corresponding closed segments determined by them.
\end{enumerate}
\end{theorem}

\begin{proof}
{\em Case $\lambda =2$}. Since $X$ is in general position, there is a point $x\in X$ which is not in $L$.  Let $H$ be the hyperplane generated by $L$ and $x$. By general position there are at most $d$ points in $H$. Since $L$ is a  transversal to the convex hulls of all $k$-sets of $X$, then at each side of $H$ there are at most $k-2$ points. The fact that $X$ has $d+2(k-2)$ points implies that $H$ has exactly $d$ points of $X$ and there are exactly $k-2$ points at each side of $H$.  Note now that $L$ is a hyperplane of $H$. If in one of the open  halfspaces of $H$ determined by $L$ there are two points of $X$, then these two points together with the $k-2$ points outside $H$, but in the same side, give rise to a $k$-set whose convex hull does not intersect $L$. This implies that either $L$ is a complete Kneser transversal or that  there are precisely $d-2$ points of $X$ in $L$. Furthermore,  there is  a unique $y\in X\cap H-L$ such that the closed segment determined $x$ and $y$ intersect $L$.  By repeating the same argument with the other $2(k-2)$ points of $X$ outside $H$, we obtain the desired conclusion.
\medskip

{\em Case $\lambda=3$}. The proof consist of the following two steps.  In the first step, we analyze the case in which  $|X\cap L| = d-3$ and conclude that this case never happens. In the second step, we analyze the case in which
$|X\cap L| \leq d-4$.  In this second case, we conclude that $|X\cap L| =d-4$ and  the other $2(k-1)$ points of $X$ are matched in $k-1$ pairs in such a way that $L$ intersects the corresponding closed segments determined by them.

{\em First step}.   Let us assume that $X\cap L=\{a_1,\dots ,a_{d-3}\}$ has exactly $d-3$ points. Let $\Theta$ be the collection of triples $\{x,y,z \}$ of $X-L$ such that the interior of the triangle with vertices  $\{x,y,z \}$ intersects the $(d-3)$-plane $L$ in exactly one point.  Note that  the intersection of any two triples of $\Theta$ contains at most one point, otherwise if $\{x,y,z_1 \}$ and $\{x,y,z_2 \}$ belong to $\Theta$, then $$\{a_1,\dots ,a_{d-3}, x,y,z_1,z_2\}$$ are $d+1$ points of $X$ contained in a plane of dimension $d-1$, contradicting the fact that $X \subset \mathbb{R}^{d}$ lies in general position.
\medskip

{\em Subcase a)} $L$ does not intersect any line generated by points of $X-L$. We shall prove that in this case  given $x \not= y \in X-L$ there is $z\in X-L$ such that the triple $\{x,y,z \}\in\Theta$. Indeed, let $x \not= y \in X-L$.  Then $L, x$ and $y$ generates a hyperplane $H^{d-1}$ of $\mathbb{R}^{d}$.  By general position there are at most $d$ points in $H^{d-1}$. Since $L$ is transversal to the convex hulls of all $k$-sets of $X$, then at each side of $H^{d-1}$ there are at most $k-3$ points. The fact that $X$ has $d+2(k-3)$ points implies that $H^{d-1}$ has exactly $d$ points of $X$ and there are exactly $k-3$ points at each side of $H$. Therefore, there is $z\in (X-L)\cap H^{d-1}$ different from $x$ and $y$. The triangle with vertices $\{x,y,z \}$ intersects $L$, otherwise the $k-3$ points of $X$ on one side of $H^{d-1}$ plus $x, y$ and $z$ give rise to a $k$-set of $X$ that avoids $L$.
Finally, the interior of the triangle with vertices  $\{x,y,z \}$ intersects  $L$ in exactly one point,
 because $L$ does not intersect any line generated by points of $X-L$. With this we have proved that given $x \not= y \in X-L$ there is $z\in X-L$ such that the triple $\{x,y,z \}\in\Theta$.
\smallskip

Let $L^\perp$ be the 3-dimensional plane through the origin orthogonal to $L$ and let $\pi:\mathbb{R}^{d}\to L^\perp$ be the orthogonal projection.  Let $X^\prime = \pi (X-L)$ and let $\Theta^\prime$ be the collection of triples of $X^\prime$ consisting of sets  $\{\pi(x),\pi(y),\pi(z) \}$ for which  $\{x,y,z\}\in \Theta$. By Lemma \ref{lem2}, $X^\prime$ lies in a two dimensional plane of $L^\perp$ and hence  $X$ lies in a hyperplane of  $\mathbb{R}^{d}$ contradicting  the fact that $X$ is in general position.
\medskip

{\em Subcase b)} Assume there are $a,b \in X-L$ such that $L$ intersects the line through $a$ and $b$. By general position  $L$ does not intersect any line generated by points of $X-(L\cup \{a,b \})$, otherwise, if $x,y\in X-(L\cup \{a,b \})$ are two points with the property that $L$ intersects the line through $x$ and $y$. Then, $\{a_1,\dots ,a_{d-3}, x,y,a,b\}$ are $d+1$ points of $X$ contained in a plane of dimension $d-1$, contradicting the fact that $X \subset \mathbb{R}^{d}$ lies in general position.

For the Subcase b),  let $\Theta$ be the collection of triples $\{x,y,z \}$ of $X-(L\cup \{a,b \})$ such that the interior of the triangle with vertices  $\{x,y,z \}$ intersects the $(d-3)$-plane $L$ in exactly one point. As above, given $x \not= y \in X-(L\cup \{a,b \})$ there is $z\in X-L$ such that the interior of the triangle with vertices  $\{x,y,z \}$ intersects $L$ in exactly one point. We  {claim} that $z\notin \{a,b\}$.  {Indeed, let us} consider $H$ the hyperplane generated by the points $(X\cap L)\cup \{x,y,z \}$. If $z\in \{a,b \}$, then $\{a,b\}\subset H$ and hence $H$ contains more than $d$ points of $X$ contradicting the general position hypothesis.
By using Lemma \ref{lem2}, as in the Subcase a) but now for $X-\{a,b\}$,  it  {follows} that $X-\{a,b\}$ lies in a hyperplane of  $\mathbb{R}^{d}$ contradicting  again the fact that $X$ is in general position.
\smallskip

{\em Second  step}. We may assume $d\geq 4$, otherwise if $d=3$ and $L$ is a non-complete Kneser transversal of dimension zero then $|X\cap L| = 0=d-3$, which is imposible by the first step. 

Suppose that $L$ has at most $d-4$ points of $X$. In this case, since $X$ is in general position, there is a point $x\in X$ which is not in $L$.  Let $H^\prime$ be the $(d-2)$-dimensional plane generated by $L$ and $x$. Again by general position, there is a point  $y\in X$ which is not in $H^\prime$. Let $H$ be the hyperplane generated by $L$ and $x$ and $y$. By our previous arguments, $H$ has exactly $d$ points of $X$ and there are exactly $k-3$ points at each side of $H$.   Note that $L$ is transversal to every triangle of  $Z=X\cap H$, otherwise, these three points together with the $k-3$ points outside $H$, but in the same side, give rise to a $k$-set whose convex hull does not intersect $L$.  
\smallskip

By the above and Lemma \ref{lem1} with $m=d-1$, $L$ has exactly  $d-4$ points of $X$ and therefore there are exactly 4 points $T=\{t_1,t_2,t_3,t_4\}$ of $X$ in $H$ which are not in $L$. Moreover, $L$ is transversal to two opposite edges of the tetrahedron $T$.  By repeating the same argument with the remaining $2(k-3)$ points of $X$ outside $H$, we obtain the desired matching. This is so because if we get say pair $uv$ in the first step, then we cannot get the pair $uw$ in the next step. The reason is the following. Let $H^\prime$ be the $(d-2)$-plane generated by $L$ and $\{u,v,w\}$. By general position there is $x\in X$ which is not in $H^\prime$.  Let $H$ be the hyperplane generated by $H^\prime$ and $x$. Consequently, $H$ contains exactly $d$ points of $X$ and as above there are exactly $k-3$ points of $X$ at each side of $H$. This implies that $L$ is transversal to every triangle of  $X\cap (H-L)$, otherwise, these three points together with the $k-3$ points outside $H$, but in the same side, give rise to a $k$-set whose convex hull does not intersect $L$.  But this is a contradiction because the triangle with vertices $\{v,w,x\}$ does not intersect $L$.
\end{proof}

 It is not difficult to generalize Theorem ~\ref{th:kt} for $\lambda>3$ if we ask $d-2(\lambda-1)$ points of $X$ to be contained in a $(d-\lambda)$-plane transversal $L$ and a collection of simplices (formed with the rest of points) of different dimensions (not necessarily intervals) to be all intersected by $L$. However,  we have not been able to prove that also for $\lambda>3$ every simplex can be chosen to be an interval.

\medskip

We may now prove Theorem~\ref{stability}.

\begin{proof}[Proof of Theorem~\ref{stability}]
Let us first introduce some notation. Let  $[n]=\{1, \dots,n\}$ and let $\binom{[n]}{s}$ be the collection of all subsets of size $s$. We denote by $\Omega$ the finite set of partitions of $[d + 2(k-\lambda)]$ of the form $[d + 2(k-\lambda)]=B_1\sqcup B_2\sqcup \dots \sqcup B_k$, where $|B_1|= d-2(\lambda-1)$ and $|B_i|= 2$, for every $2\leq i \leq k$.

Let $X$ be our collection of points that we suppose, without loss of generality, in general position (if not, we slighly perturb $X$). Let $L$ be a non-complete $(d-\lambda)$-plane transversal to the convex hulls of all $k$-sets of $X$.   {Since $X$ is in general position then, by Theorem ~\ref{th:kt}, we have that} $ |X\cap L|= d-2(\lambda-1)$ and the other $2(k-1)$ points of $X$ are matched in $k-1$ pairs in such a way that $L$ intersects the corresponding closed segments determined by them. Therefore, this non-complete Kneser transversal  $L$ naturally yields a partition $\alpha \in \Omega$.  If this is so, we shall say that $L$ is a Kneser transversal of {\em type $\alpha$}.
We have the  following immediate consequences of this definition.

\begin{enumerate}
\item
To every non-complete Kneser transversal it corresponds a unique type $\alpha\in \Omega$.
\item
By Lemma \ref{lem3},  since $k -\lambda\geq 2$, any two non-complete Kneser transversals of the same type coincide.
\item
Since $\Omega$ is a finite set then $X$ admits a finite number of Kneser transversals.
\item
Given $X$ and $\alpha\in \Omega$, there exists $\epsilon >0$ such that if $X$ does not admit a Kneser transversal of type $\alpha$ and  $X^\prime=\{x^\prime_1,\dots ,x^\prime_n\}$ is a collection of points in $\R^{d}$  in general position such that $| x_i - x^\prime _i |< \epsilon$, for every $i\in[n]$, then  $X^\prime$ does not admit a Kneser transversal of type $\alpha\in \Omega$.
\end{enumerate}

Our next goal is to show that if $X$ admits a Kneser transversal of type $\alpha\in \Omega$ and $\epsilon >0$ is given, then we can choose a collection of points in $\R^{d}$  in general position, $X^\prime=\{x^\prime_1,\dots ,x^\prime_n\}$,  such that $|x_i - x^\prime |< \epsilon$, for every $i\in[n]$ and with the property that $X^\prime$ does not admit a Kneser transversal of type $\alpha\in \Omega$. To see this, we may assume without loss of generality that
$$
\alpha= \{ 1, \dots ,d-2(\lambda-1)\}, \{d-2\lambda +3, d-2\lambda+4\}, \dots , \{d+2k+2\lambda-1, d+2k-2\lambda\}
$$
and $L$ is a $(d-\lambda)$-plane that contains $\{ x_1, \dots ,x_{d-2(\lambda-1)}\}$ and intersects the $k-1$ closed segments with extreme points
$$
\{x_{d-2\lambda +3}, x_{d-2\lambda+4}\}, \dots ,  \{x_{d+2k+2\lambda-1}, x_{d+2k-2\lambda}\}.
$$

Choose $x^{\prime}_{d+2k-2\lambda}$ such that $| x_{d+2k-2\lambda}-x^{\prime}_{d+2k-2\lambda}|<\epsilon $ and in such a way that $L$ does not intersect the line generated by $\{x_{d+2k+2\lambda-1}, x^{\prime}_{d+2k-2\lambda}\}$. Suppose now $X^\prime =X \cup \{x^{\prime}_{d+2k-2\lambda}\}-\{x_{d+2k-2\lambda}\}$ and let $L^\prime$ be a Kneser transversal of $X^\prime$ of type $\alpha$. By Lemma~\ref{lem3}, since $k-2\geq \lambda$, $L$ must be equal than $L^\prime$, which is imposible because $L$ does not intersect the line generated by $\{x_{d+2k+2\lambda-1}, x^{\prime}_{d+2k-2\lambda}\}$. Furthermore, $\epsilon>0$ is so small that if $X$ does not contain a Kneser transversal of type $\beta\in\Omega$, $X^\prime$ either. 

As a consequence of our last statement and of (1)-(4),  it follows that given $\epsilon >0$ sufficiently small,  there is a collection of points in $\R^{d}$  in general position $X^\prime=\{x^\prime_1,\dots ,x^\prime_n\}$,  such that $| x_i - x^\prime _i | < \epsilon$, for every $i\in[n]$ and with the property that the number of non-complete Kneser transversals of $X$ is bigger that the number of non-complete Kneser transversals of $X^\prime$.  Since the number of non-complete Kneser transversals of $X$ is finite, this completes the proof of our theorem.
\end{proof}

We believe that Theorem~\ref{stability} is also true for any $\lambda> 3$ but the proof needs a more difficult and complicated version of Theorem~\ref{th:kt}.

\section{Bounds for $m(k,d,\lambda)$ when $\lambda=2,3$}\label{sec:bound}

We start by proving the following upper bound

\begin{theorem}\label{th:bound}
Let $\lambda=2,3$, $k-\lambda \geq 2$ and $2(\lambda-1)\leq d\leq 2(k-2)$. Then,
$$
m(k,d,\lambda)<d+2(k-\lambda).
$$
\end{theorem}

\begin{proof}
Let $X=\{x_1,\dots ,x_n\}$ be a finite collection of points in $\R^{d}$ embedded in the moment curve on $n=d+2(k-\lambda)$ vertices. On one hand, by Theorem \ref{stability},  there exists a collection of points $X^\prime=\{x^\prime_1,\dots ,x^\prime_n\}$ in $\R^{d}$ in general position with the order type as the cyclic polytope and with the property that every $(d-\lambda)$-plane transversal to the convex hull of the $k$-sets of $X^\prime$ is complete (i.e.,  it contains  $d-\lambda +1$ points of $X^\prime)$. On the other hand, by \cite[Corollary 3.5]{CMMMR} we have that 
$$
m^*(k,d,\lambda)\le (d-\lambda+1)+\left\lfloor\left(2-\frac{\lambda-1}{\lceil \frac{d}{2}\rceil}\right)(k-1)\right\rfloor
$$
which is strictly smaller than $d+2(k-\lambda)$ when  $\lceil \frac{d}{2}\rceil <k-1$ and thus implying that $X^\prime$ does not admit  a complete transversal. Therefore, the collection of points $X^\prime$ in $\R^{d}$ does not have a $(d-\lambda)$-plane transversal to the convex hull of the $k$-sets implying that $m(k,d,\lambda)<d+2(k-\lambda)$.
\end{proof}

A better upper bound was proved by Tancer \cite{T} in the special case when $d=3$ and $k\geq 6$. In fact, it is possible to embed $2k-2$ points of the form $(t,t^2,f(t))$ in general position in $\R^3$ where $f(t)$ is a fast growing function in such a way that there is no transversal line to the convex hull of the $k$-sets. We point out that this approach works only when $d=3$.
\medskip

The following result due to Bukh, Matou\v sek and Nivasch \cite{BMN} was proved by using equivariant topology (see also \cite{BN} and \cite{MP} for some related results).

\begin{quote}
Let $X=\{x_1, x_2, \dots x_n\}$  be a collection of $n$ points in $\mathbb{R}^{d}$.  Then, there exists a codimension two affine plane $L$ and $2d-1$ hyperplanes passing through $L$ that divide $\mathbb{R}^{d}$ into $4d-2$ parts, each containing at most $\frac{n}{4d-2} + O(1)$ points of $X$.
\end{quote}

Every hyperplane $H$ through $L$ leaves at least $2d-2$ of these parts on each side of $H$. Therefore, on each side of $H$ we must have at least $$n-2d\left(\frac{n}{4d-2}+O(1)\right)=\frac{(2d-2)n}{4d-2}-2dO(1)$$ points of $X$ and thus
$$
\left\lfloor \frac{(2d-2)n}{4d-2}-2dO(1)\right\rfloor\leq \tau(n,d,2).
$$
Furthermore, if $k\geq  \frac{2dn}{4d-2}+2dO(1)$, the codimension two affine plane $L$ intersects the convex hull of every $k$-set of $X$, implying that

\begin{equation}\label{eq3}
\left\lceil \frac{(4d-2)}{2d}\left(k-2dO(1)\right)\right\rceil\leq m(k,d,2).
\end{equation}

From Equation~\eqref{eq3} and Theorem~\ref{th:bound}, we obtain
$$
2-\frac{1}{d}\leq \limsup_{k\to\infty}\frac{m(k,d,2)}{k} \leq 2.
$$
Therefore, for $d\geq3, \lambda=2$ and $k$ large enough the conjectured value \cite[Conjecture 1]{ABMR}
$m(k,d,\lambda )=d-\lambda +k+\left\lceil \frac{k}{\lambda }\right\rceil -1$ does not hold.

\section {Computational results}\label{computer}

For a general background on oriented matroid theory we refer the reader to \cite{Bjo-99}.

An \emph{abstract order type} is the relabeling class of an acyclic oriented matroid. The abstract order types of realizable oriented matroids are called {\em order types} corresponding to {\em isomorphism types} of configurations of points in the Euclidean space.
\medskip

In \cite{ABMR}, it is proved that $m(3,2,4)=6$ (i.e., there always exists a transversal line to all tetrahedra formed by any configuration of 6 points in $\R^3$) and that there is never a transversal line to all tetrahedra formed by any configuration of 8 points in general position in $\R^3$.

\par\bigskip
What about transversal lines to all tetrahedra in configurations of 7 points in $\R^3$?

\bigskip\par

We will answer this question by classifying the configurations of 7 points in $\R^3$ having a (complete) Kneser line (if any). It is known that there are $5083$ abstract order types of rank $r=4$ ($d=3$) of cardinality $n=7$ \cite{F}. Among these $5083$ abstract order types, $246$ of them are the order type of some configuration of points in general position.

\subsection{Complete Kneser transversal line}

We investigate whether there exist a complete transversal line to the tetrahedra of $E=\{x_1,\dots ,x_7\}$. For this, we first detect when the line joining $x_{i_1}$ and $x_{i_2}$ intersects the interior of the triangle $(x_{i_3},x_{i_4},x_{i_5})$. The following proposition is a direct consequence of \cite[Proposition 2.2]{CMMMR}, but we state this particular case here for self-containment.

\begin{proposition}\label{prop1}
Let $E:=\{x_1,\ldots,x_7\}$ be a set of 7 points in general position in $\R^3$ and let $\mathcal{M}=(E,B)$ its associated oriented matroid. Then, the line $(x_{i_1},x_{i_2})$ intersects the interior of the triangle $(x_{i_3},x_{i_4},x_{i_5})$ if and only if $sg(x_{i_3})=sg(x_{i_4})=sg(x_{i_5})$ in the circuit $\{x_{i_1},x_{i_2},x_{i_3},x_{i_4},x_{i_5}\}$ of $\mathcal{M}$, where $sg(x)$ stands for the sign of $x$ in $\mathcal{M}$.
\end{proposition}

\begin{proof}
Since $\mathcal{M}$ is acyclic, it is not possible that the elements of the circuit $C:=\{x_{i_1},x_{i_2},x_{i_3},x_{i_4},x_{i_5}\}$ have all the same sign. Moreover, as depicted in Figure~\ref{fig1}, the line $(x_{i_1},x_{i_2})$ intersects the interior of the triangle $(x_{i_3},x_{i_4},x_{i_5})$ if and only if the Radon partition associated to $C$ is one of $\left\{ \left\{ x_{i_1},x_{i_2} \right\} , \left\{ x_{i_3},x_{i_4},x_{i_5} \right\} \right\}$ , $\left\{ \left\{ x_{i_1} \right\} , \left\{ x_{i_2},x_{i_3},x_{i_4},x_{i_5} \right\} \right\}$ or
$\left\{ \left\{ x_{i_2} \right\} , \left\{ x_{i_1},x_{i_3},x_{i_4},x_{i_5} \right\} \right\}$.
\end{proof}

\begin{figure}
\begin{center}
\begin{tabular}{cc}
\begin{tikzpicture}
\draw (-0.1,0.1) -- (0.1,-0.1);
\draw (0.1,0.1) -- (-0.1,-0.1);
\draw (-0.1,2.6) -- (0.1,2.4);
\draw (0.1,2.6) -- (-0.1,2.4);
\draw (-0.1,-2.4) -- (0.1,-2.6);
\draw (0.1,-2.4) -- (-0.1,-2.6);
\draw (0,0) -- (0,3);
\draw[dotted] (0,0) -- (0,-5/3);
\draw (0,-5/3) -- (0,-3);
\draw (-3,0) -- (2,1) -- (-1,-3) -- (-3,0);
\node[right] (1) at (0,2.5) {$x_{i_1}$};
\node[right] (2) at (0,-2.5) {$x_{i_2}$};
\node[left] (3) at (-3,0) {$x_{i_3}$};
\node[right] (4) at (2,1) {$x_{i_4}$};
\node[below] (5) at (-1,-3) {$x_{i_5}$};
\end{tikzpicture}
&
\begin{tikzpicture}
\draw (-0.1,0.1) -- (0.1,-0.1);
\draw (0.1,0.1) -- (-0.1,-0.1);
\draw (-0.1,2.6) -- (0.1,2.4);
\draw (0.1,2.6) -- (-0.1,2.4);
\draw (-0.1,1.35) -- (0.1,1.15);
\draw (0.1,1.35) -- (-0.1,1.15);
\draw (0,0) -- (0,3);
\draw[dotted] (0,0) -- (0,-5/3);
\draw (0,-5/3) -- (0,-3);
\draw (-3,0) -- (2,1) -- (-1,-3) -- (-3,0);
\node[right] (1) at (0,2.5) {$x_{i_1}$};
\node[right] (2) at (0,1.25) {$x_{i_2}$};
\node[left] (3) at (-3,0) {$x_{i_3}$};
\node[right] (4) at (2,1) {$x_{i_4}$};
\node[below] (5) at (-1,-3) {$x_{i_5}$};
\end{tikzpicture} \\
$\left\{ \left\{ x_{i_1},x_{i_2} \right\} , \left\{ x_{i_3},x_{i_4},x_{i_5} \right\} \right\}$ & $\left\{ \left\{ x_{i_2} \right\} , \left\{ x_{i_1},x_{i_3},x_{i_4},x_{i_5} \right\} \right\}$
\end{tabular}
\end{center}
\caption{Radon partitions where the line $(x_{i_1},x_{i_2})$ intersects the triangle $(x_{i_3},x_{i_4},x_{i_5})$.}\label{fig1}
\end{figure}
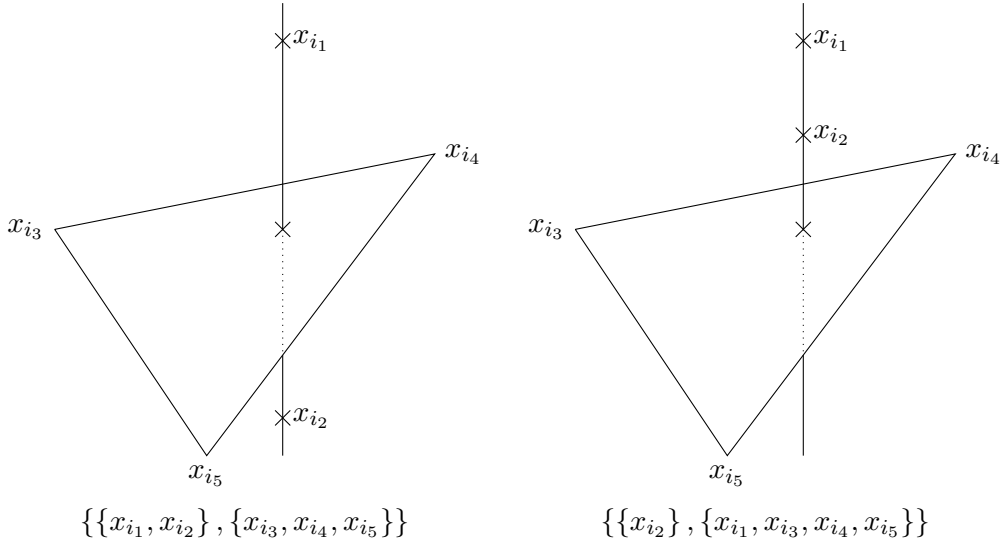

We notice that since the points of $E$ are in general position, then the line $(x_{i_1},x_{i_2})$ cannot intersect the triangle $(x_{i_3},x_{i_4},x_{i_5})$ on a vertex or an edge.

For each of the $\binom{7}{2}=21$ pairs $(x_{i_1},x_{i_2})$, we determine, by using Proposition~\ref{prop1}, if the line $(x_{i_1},x_{i_2})$ intersects the $\binom{5}{3}=10$ triangles of $E\setminus\{x_{i_1},x_{i_2}\}$. Since the points are in general position, it is easy to see that if $(x_{i_1},x_{i_2})$ intersects a tetrahedron $T$ whose vertices are in $E$, then $(x_{i_1},x_{i_2})$ intersects at least two faces (two triangles) of $T$ (c.f. \cite[Proposition 2.1]{CMMMR}). Therefore, if $(x_{i_1},x_{i_2})$ intersects the triangle $(x_{i_3},x_{i_4},x_{i_5})$, it intersects both tetrahedra $(x_{i_3},x_{i_4},x_{i_5},x_{i_6})$ and $(x_{i_3},x_{i_4},x_{i_5},x_{i_7})$. Finally, if the line $(x_{i_1},x_{i_2})$ intersects the $\binom{5}{4}=5$ tetrahedra generated from $E\setminus\{x_{i_1},x_{i_2}\}$, it immediately follows that $(x_{i_1},x_{i_2})$ is transversal to all the tetrahedra of $E$.

For instance, for $\mathcal{M}=OT(7,4,2)$ in the classification given in \cite{F}, that is the abstract order type representing a point configuration having the chirotope $\chi_{\mathcal{M}} : \mathcal{B} \rightarrow \{0,-,+\}$ given by
$$
\scriptsize
\begin{array}{ccccccccccccccccccc}
 & 1 & 1 & 1 & 1 & 2 & 1 & 1 & 1 & 2 & 1 & 1 & 2 & 1 & 2 & 3 & 1 & 1 & 1  \\
 & 2 & 2 & 2 & 3 & 3 & 2 & 2 & 3 & 3 & 2 & 3 & 3 & 4 & 4 & 4 & 2 & 2 & 3  \\
 & 3 & 3 & 4 & 4 & 4 & 3 & 4 & 4 & 4 & 5 & 5 & 5 & 5 & 5 & 5 & 3 & 4 & 4  \\
 & 4 & 5 & 5 & 5 & 5 & 6 & 6 & 6 & 6 & 6 & 6 & 6 & 6 & 6 & 6 & 7 & 7 & 7  \\
 \chi_{\mathcal{M}} = & + & - & + & - & - & - & + & - & - & - & + & + & - & - & + & + & - & +  \\
 \\
  & 2 & 1 & 1 & 2 & 1 & 2 & 3 & 1 & 1 & 2 & 1 & 2 & 3 & 1 & 2 & 3 & 4 & \\
 & 3 & 2 & 3 & 3 & 4 & 4 & 4 & 2 & 3 & 3 & 4 & 4 & 4 & 5 & 5 & 5 & 5 & \\
 & 4 & 5 & 5 & 5 & 5 & 5 & 5 & 6 & 6 & 6 & 6 & 6 & 6 & 6 & 6 & 6 & 6 & \\
 & 7 & 7 & 7 & 7 & 7 & 7 & 7 & 7 & 7 & 7 & 7 & 7 & 7 & 7 & 7 & 7 & 7 & \\
 \chi_{\mathcal{M}} = & + & + & - & - & + & + & + & + & - & - & + & - & + & + & + & - & +& \\
\end{array}
$$
the line $L$ going through $1$ and $5$ is a complete Kneser transversal line. Indeed, by Proposition~\ref{prop1}, we know that $L$ intersects the triangles $(2,3,4)$, $(2,3,6)$, $(2,6,7)$ and $(3,4,7)$ since the corresponding circuits are
$$
\begin{array}{cccc}
\{\bar{1}2345\}, & \{1\bar{2}\bar{3}5\bar{6}\}, & \{1\bar{2}5\bar{6}\bar{7}\}, & \{1\bar{3}\bar{4}\bar{5}\bar{7}\},\\
\end{array}
$$
implying  that $L$ intersects the $5$ tetrahedra $(2,3,4,6)$, $(2,3,4,7)$, $(2,3,6,7)$, $(2,4,6,7)$ and $(3,4,6,7)$.
\medskip

By applying the above method, we obtain the following result.

\begin{theorem}\label{th:comp1}
Among the $246$ configurations of $7$ points in general position in $\R^3$ there are $124$ admitting a complete Kneser transversal to the tetrahedra. These configurations correspond to the $124$ realizable rank $4$ oriented matroids on $7$ elements given by the following set according to the classification in \cite{F}
{\scriptsize $$
A:=
\begin{array}[t]{l}
\left\{2,3,5,6,8,9,10,15,16,18,20,21,25,27,28,29,33,34,35,38,40,41,43,44,45,46,\right.\\
47,48,50,51,52,55,56,60,62,63,64,67,68,69,70,71,72,74,76,79,85,88,92,93,\\
94,95,96,97,98,99,100,101,102,106,107,109,110,111,112,113,118,120,123,124,\\
125,127,132,134,135,136,140,141,142,144,145,150,151,154,155,156,157,159,\\
160,166,167,171,172,177,178,182,183,184,185,186,187,189,191,192,195,199,\\
\left. 200,201,206,207,208,211,212,219,220,221,224,225,228,229,234,237,243,244\right\}.
\end{array}
$$}
\end{theorem}

\subsection{Kneser transversal line}

Let $E:=\{x_1,\ldots,x_7\}$ be a set of 7 points in general position in $\R^3$ and let $\mathcal{M}=(E,\mathcal{B})$ be its associated oriented matroid. By Theorem \ref{th:kt}, if there exist a non-complete Kneser transversal line to the convex hull of its $4$-subsets, then the $7$ points of $\mathcal{M}$ must look as depicted in Figure~\ref{fig2}. This implies that $\mathcal{M}$ admits the following circuits

\begin{eqnarray}\label{eqa}
\{\overline{x_1}\, \overline{x_2}\, x_3\, x_4\, x_7\},
&
\{\overline{x_1}\, \overline{x_2}\, x_5\, x_6\, x_7\},
&
\{\overline{x_3}\, \overline{x_4}\, x_5\, x_6\, x_7\},
\end{eqnarray}

and cocircuits
\begin{eqnarray}\label{eqb}
\{\overline{x_3}\, x_4\, \overline{x_5}\, x_6\},
&
\{\overline{x_1}\, x_2\, x_5\, \overline{x_6}\},
&
\{\overline{x_1}\, x_2\, \overline{x_3}\, x_4\}.
\end{eqnarray}

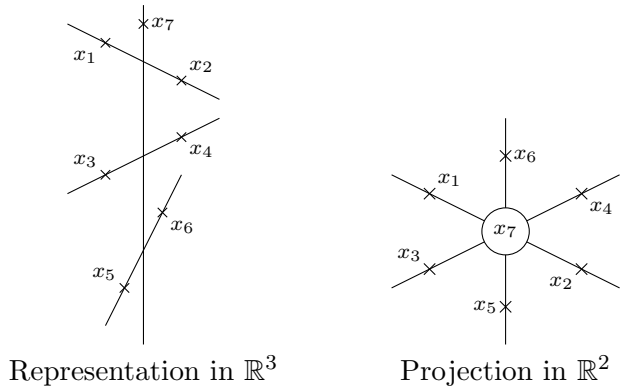
\begin{figure}
\begin{center}
\begin{tabular}{ccc}
\begin{tikzpicture}[scale=0.25]
\scriptsize
\draw (0,0) -- (0,18);
\draw (-2,1) -- (2,9);
\draw (-4,8) -- (4,12);
\draw (-4,17) -- (4,13);

\draw (-2.25,16.25) -- (-1.75,15.75);
\draw (-1.75,16.25) -- (-2.25,15.75);
\node[below left] (1) at (-2,16) {$x_1$};

\draw (1.75,14.25) -- (2.25,13.75);
\draw (2.25,14.25) -- (1.75,13.75);
\node[above right] (2) at (2,14) {$x_2$};

\draw (-2.25,9.25) -- (-1.75,8.75);
\draw (-1.75,9.25) -- (-2.25,8.75);
\node[above left] (3) at (-2,9) {$x_3$};

\draw (1.75,11.25) -- (2.25,10.75);
\draw (2.25,11.25) -- (1.75,10.75);
\node[below right] (4) at (2,11) {$x_4$};

\draw (-1.25,3.25) -- (-0.75,2.75);
\draw (-0.75,3.25) -- (-1.25,2.75);
\node[above left] (5) at (-1,3) {$x_5$};

\draw (0.75,7.25) -- (1.25,6.75);
\draw (1.25,7.25) -- (0.75,6.75);
\node[below right] (6) at (1,7) {$x_6$};

\draw (-0.25,17.25) -- (0.25,16.75);
\draw (0.25,17.25) -- (-0.25,16.75);
\node[right] (7) at (0,17) {$x_7$};
\end{tikzpicture}
& \quad\quad\quad &
\begin{tikzpicture}[scale=0.5]
\scriptsize
\draw (0,0) -- (0,6);
\draw (-3,1.5) -- (3,4.5);
\draw (3,1.5) -- (-3,4.5);

\draw (-2.15,4.15) -- (-1.85,3.85);
\draw (-1.85,4.15) -- (-2.15,3.85);
\node[above right] (1) at (-2,4) {$x_1$};

\draw (1.85,2.15) -- (2.15,1.85);
\draw (2.15,2.15) -- (1.85,1.85);
\node[below left] (2) at (2,2) {$x_2$};

\draw (-2.15,2.15) -- (-1.85,1.85);
\draw (-1.85,2.15) -- (-2.15,1.85);
\node[above left] (3) at (-2,2) {$x_3$};

\draw (1.85,4.15) -- (2.15,3.85);
\draw (2.15,4.15) -- (1.85,3.85);
\node[below right] (4) at (2,4) {$x_4$};

\draw (-0.15,1.15) -- (0.15,0.85);
\draw (0.15,1.15) -- (-0.15,0.85);
\node[left] (5) at (0,1) {$x_5$};

\draw (-0.15,5.15) -- (0.15,4.85);
\draw (0.15,5.15) -- (-0.15,4.85);
\node[right] (6) at (0,5) {$x_6$};

\node[draw,circle,fill=white] (7) at (0,3) {$x_7$};
\end{tikzpicture} \\
Representation in $\R^3$ & & Projection in $\R^2$
\end{tabular}
\end{center}
\caption{$7$ points in $\R^3$ with circuits and cocircuits satisfying \eqref{eqa} and \eqref{eqb} with a Kneser transversal line to all $4$-sets}\label{fig2}
\end{figure}

However it is possible that certain configurations of $7$ points $\mathcal{M}$ having circuits as given in \eqref{eqa} and cocircuits as in \eqref{eqb} do not admit a transversal line to the convex hull of its $4$-subsets, see Figure~\ref{fig3} where this situation is illustrated. This example shows once more that the existence of Kneser Transversals is not an invariant of the order type in general.

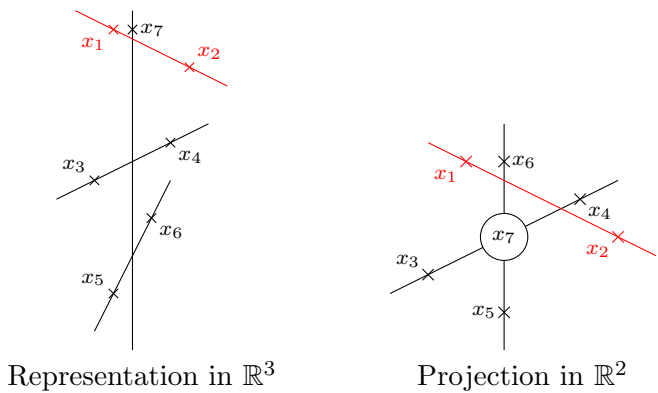
\begin{figure}
\begin{center}
\begin{tabular}{ccc}
\begin{tikzpicture}[scale=0.25]
\scriptsize
\draw (0,0) -- (0,18);
\draw (-2,1) -- (2,9);
\draw (-4,8) -- (4,12);
\draw[color=red] (-3,18) -- (5,14);

\draw[color=red] (-1.25,17.25) -- (-0.75,16.75);
\draw[color=red] (-0.75,17.25) -- (-1.25,16.75);
\node[color=red,below left] (1) at (-1,17) {$x_1$};

\draw[color=red] (2.75,15.25) -- (3.25,14.75);
\draw[color=red] (3.25,15.25) -- (2.75,14.75);
\node[color=red,above right] (2) at (3,15) {$x_2$};

\draw (-2.25,9.25) -- (-1.75,8.75);
\draw (-1.75,9.25) -- (-2.25,8.75);
\node[above left] (3) at (-2,9) {$x_3$};

\draw (1.75,11.25) -- (2.25,10.75);
\draw (2.25,11.25) -- (1.75,10.75);
\node[below right] (4) at (2,11) {$x_4$};

\draw (-1.25,3.25) -- (-0.75,2.75);
\draw (-0.75,3.25) -- (-1.25,2.75);
\node[above left] (5) at (-1,3) {$x_5$};

\draw (0.75,7.25) -- (1.25,6.75);
\draw (1.25,7.25) -- (0.75,6.75);
\node[below right] (6) at (1,7) {$x_6$};

\draw (-0.25,17.25) -- (0.25,16.75);
\draw (0.25,17.25) -- (-0.25,16.75);
\node[right] (7) at (0,17) {$x_7$};
\end{tikzpicture}
& \quad\quad\quad &
\begin{tikzpicture}[scale=0.5]
\scriptsize
\draw (0,0) -- (0,6);
\draw (-3,1.5) -- (3,4.5);
\draw[color=red] (4,2.5) -- (-2,5.5);

\draw[color=red] (-1.15,5.15) -- (-0.85,4.85);
\draw[color=red] (-0.85,5.15) -- (-1.15,4.85);
\node[color=red,below left] (1) at (-1,5) {$x_1$};

\draw[color=red] (2.85,3.15) -- (3.15,2.85);
\draw[color=red] (3.15,3.15) -- (2.85,2.85);
\node[color=red,below left] (2) at (3,3) {$x_2$};

\draw (-2.15,2.15) -- (-1.85,1.85);
\draw (-1.85,2.15) -- (-2.15,1.85);
\node[above left] (3) at (-2,2) {$x_3$};

\draw (1.85,4.15) -- (2.15,3.85);
\draw (2.15,4.15) -- (1.85,3.85);
\node[below right] (4) at (2,4) {$x_4$};

\draw (-0.15,1.15) -- (0.15,0.85);
\draw (0.15,1.15) -- (-0.15,0.85);
\node[left] (5) at (0,1) {$x_5$};

\draw (-0.15,5.15) -- (0.15,4.85);
\draw (0.15,5.15) -- (-0.15,4.85);
\node[right] (6) at (0,5) {$x_6$};

\node[draw,circle,fill=white] (7) at (0,3) {$x_7$};
\end{tikzpicture}\\
Representation in $\R^3$ & & Projection in $\R^2$
\end{tabular}
\end{center}
\caption{$7$ points in $\R^3$ with circuits and cocircuits satisfying \eqref{eqa} and \eqref{eqb} but without transversal line to all $4$-sets}\label{fig3}
\end{figure}

Nevertheless, we can still identify whether a configuration of 7 points admits a Kneser transversal line. To this end, we consider the oriented matroid $\mathcal{M}'$ associated to the configuration of  $8$ points $E':=\{x_1,\dots,x_8\}$ in $\R^3$, not necessarily in general position, illustrated in Figure~\ref{fig4}. The deletion of either point $x_7$ or point $x_8$ from $\mathcal{M}$ yields a configuration on $7$ points as represented in Figure~\ref{fig2} admitting thus a Kneser transversal line (containing either $x_7$ or $x_8$) to the all tetrahedra. We thus have that the line going through $x_7$ and $x_8$ would be a complete Kneser transversal line of $E'$. Moreover, any  configuration on $7$ points as represented in Figure~\ref{fig2} arises on this way.

We may thus detect all such configurations $\mathcal{M}'$. We do this by observing that an oriented matroid $\mathcal{M}$ on 8 elements correspond to such configuration if and only if $\mathcal{M}$ admits the following cocircuits

\begin{eqnarray}\label{eqc}
\{\overline{x_3}\, x_4\, \overline{x_5}\, x_6\},
&
\{\overline{x_1}\, x_2\, x_5\, \overline{x_6}\},
&
\{\overline{x_1}\, x_2\, \overline{x_3}\, x_4\}.
\end{eqnarray}

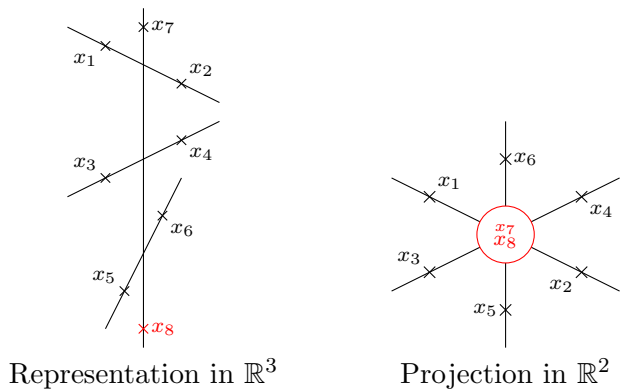
\begin{figure}
\begin{center}
\begin{tabular}{ccc}
\begin{tikzpicture}[scale=0.25]
\scriptsize
\draw (0,0) -- (0,18);
\draw (-2,1) -- (2,9);
\draw (-4,8) -- (4,12);
\draw (-4,17) -- (4,13);

\draw (-2.25,16.25) -- (-1.75,15.75);
\draw (-1.75,16.25) -- (-2.25,15.75);
\node[below left] (1) at (-2,16) {$x_1$};

\draw (1.75,14.25) -- (2.25,13.75);
\draw (2.25,14.25) -- (1.75,13.75);
\node[above right] (2) at (2,14) {$x_2$};

\draw (-2.25,9.25) -- (-1.75,8.75);
\draw (-1.75,9.25) -- (-2.25,8.75);
\node[above left] (3) at (-2,9) {$x_3$};

\draw (1.75,11.25) -- (2.25,10.75);
\draw (2.25,11.25) -- (1.75,10.75);
\node[below right] (4) at (2,11) {$x_4$};

\draw (-1.25,3.25) -- (-0.75,2.75);
\draw (-0.75,3.25) -- (-1.25,2.75);
\node[above left] (5) at (-1,3) {$x_5$};

\draw (0.75,7.25) -- (1.25,6.75);
\draw (1.25,7.25) -- (0.75,6.75);
\node[below right] (6) at (1,7) {$x_6$};

\draw (-0.25,17.25) -- (0.25,16.75);
\draw (0.25,17.25) -- (-0.25,16.75);
\node[right] (7) at (0,17) {$x_7$};

\draw[color=red] (-0.25,1.25) -- (0.25,0.75);
\draw[color=red] (0.25,1.25) -- (-0.25,0.75);
\node[color=red,right] (8) at (0,1) {$x_8$};
\end{tikzpicture}
& \quad\quad\quad &
\begin{tikzpicture}[scale=0.5]
\scriptsize
\draw (0,0) -- (0,6);
\draw (-3,1.5) -- (3,4.5);
\draw (3,1.5) -- (-3,4.5);

\draw (-2.15,4.15) -- (-1.85,3.85);
\draw (-1.85,4.15) -- (-2.15,3.85);
\node[above right] (1) at (-2,4) {$x_1$};

\draw (1.85,2.15) -- (2.15,1.85);
\draw (2.15,2.15) -- (1.85,1.85);
\node[below left] (2) at (2,2) {$x_2$};

\draw (-2.15,2.15) -- (-1.85,1.85);
\draw (-1.85,2.15) -- (-2.15,1.85);
\node[above left] (3) at (-2,2) {$x_3$};

\draw (1.85,4.15) -- (2.15,3.85);
\draw (2.15,4.15) -- (1.85,3.85);
\node[below right] (4) at (2,4) {$x_4$};

\draw (-0.15,1.15) -- (0.15,0.85);
\draw (0.15,1.15) -- (-0.15,0.85);
\node[left] (5) at (0,1) {$x_5$};

\draw (-0.15,5.15) -- (0.15,4.85);
\draw (0.15,5.15) -- (-0.15,4.85);
\node[right] (6) at (0,5) {$x_6$};

\node[color=red,draw,circle,fill=white] (7) at (0,3) {$\stackrel{x_7}{x_8}$};
\end{tikzpicture} \\
Representation in $\R^3$ & & Projection in $\R^2$
\end{tabular}
\end{center}
\caption{$8$ points in $\R^3$ with a complete transversal line to all $4$-sets}\label{fig4}
\end{figure}

For each of the $10775236$ order types of $8$ points in $\R^3$, we consider its associated oriented matroid $\mathcal{M}$. If $\mathcal{M}$ admits cocircuits of the form \eqref{eqc}, we delete $x_7$ or $x_8$ obtaining a configuration of $7$ points in general position as in Figure~\ref{fig2} admitting a non-complete Kneser transversal line to all tetrehedra.

\begin{theorem}\label{th:comp2}
Among the $246$ different order types of $7$ points in general position in $\R^3$ there are $124$ admitting a representation for which there is a non-complete Kneser transversal line to all tetrahedra. These configurations correspond to the $124$ realizable rank $4$ oriented matroids on $7$ elements given in by the following set according to the classification in \cite{F}
{\scriptsize $$
B :=
\begin{array}[t]{l}
\left\{1,2,4,6,7,10,11,12,13,14,16,19,24,26,29,30,31,32,36,38,39,40,41,42,55,58,\right.\\
59,60,61,62,63,65,70,71,72,74,76,77,78,79,80,81,82,83,85,86,87,88,89,90,\\
91,93,96,97,98,99,101,102,103,104,105,111,114,117,121,122,124,126,128,129,\\
130,138,139,140,145,146,147,148,149,151,152,153,158,165,170,171,172,173,\\
174,175,176,177,180,185,194,196,197,198,199,201,204,206,207,209,211,212,\\
\left. 213,214,215,217,218,219,220,230,235,236,237,238,239,240,241,242,244,246\right\}.
\end{array}
$$}
\end{theorem}

\begin{corollary}
Among the $246$ configurations of $7$ points in general position in $\R^3$, $124$ of them admit a Kneser transversal line to all the tetrahedra
and $122$ configurations do not admit a Kneser transversal line to the all the tetrahedra.
\end{corollary}

\begin{proof} By Theorems \ref{th:comp1} and \ref{th:comp2}, we have
$$
|A|=124, |B|=124, |A\cap B|=46, |A\setminus B|=78, |B\setminus A|=78, |\overline{A\cup B}|=44.
$$
The $44$ order types of $\overline{A\cup B}$ do not admit Kneser transversal lines. By Theorem \ref{stability}, for each of the $78$ order types of $B\setminus A$, there exists a representation for which there is no Kneser transversal line.
\end{proof}

It turns out that all abstract order types of rank 4 with at most 7 elements are realizable \cite[Corollary 8.3.3]{Bjo-99}, that is, they correspond to 7 points in the space, but not necessarily in general position.

\begin{theorem}
Among the $5083$ abstract order types of rank $r=4$ ($d=3$) with $n=7$ there are $1158$ having a complete Kneser transversal line to all the tetrahedra.
\end{theorem}

This calculation was obtained by applying the same arguments as those used for Theorem \ref{th:comp1} (via Radon partitions) and by taking care of the cases when three or more points are collinear and when four of more points are coplanar (via the circuits of the the oriented matroid).
 
\section*{Acknowledgments}
The authors would like to thank the anonymous referees for their useful comments and remarks.



\begin{thebibliography}{}

\bibitem{ABMR} J.L. Arocha, J. Bracho, L. Montejano and J.L. Ram\'{\i}rez Alfons\'{\i}n, Transversals to the convex hulls of all $k$-sets of discrete subsets of $\mathbb{R}^{n}$, \emph{J. Combin. Theory Ser. A} \textbf{118}{}(1) (2010), 197--207.

\bibitem{Bjo-99} A. Bj{\"o}rner, M. Las Vergnas, B. Sturmfels, N. White, G. Ziegler, {Oriented matroids}, {\em Encyclopedia of Mathematics and its Applications, {\bf 46},
(second edition), Cambridge University Press} (1999), xii+548p.

\bibitem{BMN} B. Bukh, J. Matou\v sek and G. Nivasch,  Stabbing simplices by points and flats, \emph{Discrete and Computational Geometry} \textbf{43}{}(2) (2010), 321--338.

\bibitem{BN} B. Bukh and G. Nivasch,  Upper bounds for centerlines, {\em Journal of Computational Geometry} {\bf 3}(1) (2012), 20-30.

\bibitem{CMMMR} J. Chappelon, L. Mart\'inez-Sandoval, L. Montejano, L.P. Montejano and J.L. Ram\'{\i}rez Alfons\'{\i}n, Complete Kneser Transversals, \emph{Adv. Appl. Math.} \textbf{82} (2017), 83--101.

\bibitem{D} V. L. Dol'nikov, On transversals of families of convex sets,  in \emph{Research in Theory of Functions of Several Real Variables,}  Yaroslavl State University, (1981), 30-36 (in Russian).

\bibitem{F} L. Finschi, Catalog of Oriented Matroids, http://www.om.math.ethz.ch/

\bibitem{L} L. Lov\'asz, Kneser conjecture, chromatic number and homotopy, \emph{J. Combin. Theory Ser. A} \textbf{25} (1978), 319--324.

\bibitem{MP} A. Magazinov and A. P\'or, An improvement on the trivial lower bound for the depth of a centerline,  Arxiv.:1603.01641v1 (2016).

\bibitem{Rado} R.Rado, A theorem on general measure, \emph{J.\ London Math. Soc.} \textbf{22} (1947), 291--300.

\bibitem{T} M. Tancer,  \emph{Personal communication}.

\end{thebibliography}
\end{document}